\documentclass[11pt,reqno]{amsart}
\usepackage{amsmath}
\usepackage{amssymb}
\usepackage{amsthm}
\usepackage{graphicx}
\usepackage{mathrsfs}
\graphicspath{{picture/}}
\usepackage[margin=1in]{geometry}
\usepackage[colorlinks=true,allcolors=black]{hyperref}
\usepackage{enumerate}
\usepackage{eufrak}
\usepackage{tikz}
\usepackage{indentfirst}
\usepackage{mathabx}
\numberwithin{equation}{section}
\title[DECAY ESTIMATE FOR 3D CUBIC NLS]{A DECAY ESTIMATE FOR CUBIC DEFOCUSING NON-LINEAR SCHR\"ODINGER EQUATION IN THREE DIMENSIONS}
\author{YI SUN}
\address[Y. Sun]{Academy of Mathematics and Systems Science, The Chinese Academy of Sciences, Beijing 100190, CHINA}
\email{sunyi@amss.ac.cn}
\newtheorem{thm}{Theorem}
\numberwithin{thm}{section}
\newtheorem{prop}{Proposition}
\numberwithin{prop}{section}
\newtheorem{lem}{Lemma}
\numberwithin{lem}{section}

\numberwithin{cor}{section}
\theoremstyle{remark}
\newtheorem{rmk}{Remark}
\theoremstyle{definition}
\numberwithin{rmk}{section}
\renewcommand{\i}{\mathrm{i}}
\newcommand{\e}{\mathrm{e}}
\renewcommand{\H}{\mathrm{H}}
\renewcommand{\L}{\mathrm{L}}
\newcommand{\B}{\mathrm{B}}

\newcommand{\I}{\mathrm{I}}

\makeatletter
\@namedef{subjclassname@2020}{\textup{2020} Mathematics Subject Classification}
\makeatother

\begin{document}
	\begin{abstract}
		In this short note, we prove a decay estimate for non-linear solutions of the 3D cubic defocusing non-linear Schr\"odinger equation.
	\end{abstract}
	\maketitle
	
	\section{Introduction}
	\subsection{Background and main results}
	We consider the cubic defocusing non-linear Schr\"odinger equation in $\mathbb{R}^3$,
	\begin{equation}\label{nls}
		\left\{\begin{aligned}
			\i\partial_t u + \Delta u &= |u|^2u,\quad t\in [0,\infty),\ x\in \mathbb{R}^3,\\
			u(0,x) &= u_0(x)\in \dot{\H}^{1/2}_{x}(\mathbb{R}^3)\cap \L^1_x(\mathbb{R}^3).
		\end{aligned}\right.
	\end{equation}
	$u(t,x):[0,\infty)\times \mathbb{R}^3\to \mathbb{C}$ is a complex-valued function. Our particular choice of non-linearity makes the equation $\smash{\dot{\H}^{1/2}_x}$-\textit{critical}: If $u(t,x)$ solves \eqref{nls}, 
	$$u_{\lambda}(t,x) := \lambda u(\lambda^2 t,\lambda x),\quad \lambda>0,$$
	also solves \eqref{nls} and the $\smash{\dot{\H}^{1/2}_x}$-norm is invariant under scaling.
	
	Also the equation preserves the \textit{mass} of solutions
	$$M[u(t)] = \int_{\mathbb{R}^3} |u(t,x)|^2\ \mathrm{d}x.$$
	
	Dispersive estimates play a fundamental role in the study of non-linear dispersive equations. It is well known that for the linear propagator $\e^{\i t\Delta}$ of the Schr\"odinger equation in $\mathbb{R}^d$, one has
	\begin{equation}\label{dispersive-estimate}
		\|\e^{\i t\Delta}u_0\|_{\L^{\infty}_x}\lesssim t^{-\frac{d}{2}}\|u_0\|_{\L^1_x},\quad t>0.
	\end{equation}
	
	Meanwhile, for defocusing NLS, many scattering type results have been widely studied. Scattering means the solution $u(t)$ to the defocusing non-linear equation asymptotically behaves like a linear solution. This means one can find $u^+_0$ in a certain Sobolev space, such that
	\begin{equation}\label{general-scattering}
		\|u(t) - \e^{\i t\Delta}u^{+}_0\|_{\dot{\H}^s_x}\to 0,\quad \text{as $t\to \infty$}.
	\end{equation}
	
	A natural question is in what sense one can recover estimate \eqref{dispersive-estimate} for solutions to the non-linear equations. In earlier works, \cite{F-Z-2021, F-Z-2023} by Fan and Zhao, and \cite{G-H-S-2023} by Guo, Huang and Song, their results need the initial data to have extra regularity. In works \cite{F-K-V-Z-2025} by Fan, Killip, Visan and Zhao, and \cite{K-2025} by Kowalski, they consider the mass critical non-linear Schr\"odinger equations and energy critical non-linear Schr\"odinger equations respectively.  It is of interest to extend these results to the inter-critical case. This note aims to study the inter-critical case, by considering the three-dimensional cubic defocusing non-linear Schr\"odinger equation which is a typical model. Early results can go back to \cite{L-S-1978} by Lin and Strauss. In \cite{F-S-Z-2024}, Fan, Staffilani and Zhao established pointwise decay estimates for this equation and weakened the regularity requirement on the initial data. However, their results do not reach the scaling-critical regularity.
	
	We note that both the mass critical and energy critical cases have unconditional well-posedness results for scattering solutions. For the three-dimensional defocusing energy-critical case, see \cite{C-K-S-T-T-2008,K-V-2012}. However, the unconditional well-posedness of scattering solutions for the cubic defocusing nonlinear Schr\"odinger equation in three dimensions at the scaling-critical regularity remains open. Only the conditional result in \cite{K-M-2010} is available in this setting. Remarkably, the methods in \cite{F-K-V-Z-2025}, \cite{K-2025} and this note cannot totally handle the higher-dimensional endpoint cases where the initial data only satisfy $u_0 \in \dot{\H}^{s_c}_x \cap \L^1_x$. Such endpoint problems are very interesting and remain open.
	
	Nevertheless, the ideas in \cite{F-K-V-Z-2025} and \cite{K-2025} are also very insightful. This note combines the optimizing procedure in \cite{F-K-V-Z-2025} with the idea in \cite{K-2025} of placing the initial data in scaling-critical Besov space. Our model \eqref{nls} is $\dot{\H}^{1/2}_x$-critical. By placing the initial data in the scaling-critical Besov space $\dot{\B}^{1/2}_{2,1}$ which is stronger than scaling critical Sobolev space $\smash{\dot{\H}^{1/2}_x}$, we can recover estimate \eqref{dispersive-estimate} for solutions to \eqref{nls}. The treatment of the Duhamel term on the time interval $[0,t/2]$ in this note is inspired by the optimizing procedure in \cite{F-K-V-Z-2025}, instead of adopting the integrability in Lorentz spaces used in \cite{K-2025}. The treatment of the Duhamel term on the time interval $[t/2,t]$ essentially follows that in \cite{K-2025}. However, compared to \cite{K-2025}, the induction procedure has been modified. In \cite{K-2025}, the well-posedness of scattering solutions is determined only by the norms of the initial data. However, the main result of this note, Theorem \ref{main-result}, relies on the conditional well-posedness result in \cite{K-M-2010} by Kenig and Merle, which requires the additional assumption $\sup_{t \in (0,\infty)} \smash{\|u(t)\|_{\dot{\H}^{1/2}_x}} < \infty$. To overcome this issue, for initial data with smaller critical Besov norm, the well-posedness of scattering solutions follows from the classical perturbation theory, and then the induction still proceeds. As for Theorem \ref{main-result-2}, the induction argument is an analogue of that in \cite{K-2025}.
	
	Global well-posedness and scattering results in works \cite{D-2023} by Dodson, and in \cite{K-M-2010} by Kenig and Merle are critical to this work. We summarize the well-posedness results that are needed:
	\begin{prop}[Theorem 1.1 in \cite{K-M-2010}]\label{wellposed-result}
		Suppose that $u(t)$ solves \eqref{nls} with initial data $u_0$. Assume that
		\begin{equation}\label{bdd-condition}
			\sup_{t\in (0,+\infty)}\|u(t)\|_{\dot{\H}^{1/2}_x} = A < +\infty.
		\end{equation}
		Then there exists $u_0^+$ such that
		\begin{equation}\label{linear-like}
			\lim_{t\to+\infty}\left\|u-\e^{\i t\Delta}u_0^+\right\|_{\dot{\H}^{1/2}_x} = 0,
		\end{equation}
		and
		\begin{equation}\label{scattering}
			\int_{0}^{+\infty}\int_{\mathbb{R}^3} |u(t,x)|^5\ \mathrm{d}x\mathrm{d}t\leq g(A),
		\end{equation}
		where $g:[0,\infty)\to [0,\infty)$ is an increasing function.
	\end{prop}
	\begin{rmk}\label{rmk-on-wp-1}
		Property \eqref{linear-like} is known as \textit{scattering}, which states that the solution $u$ to the non-linear equation asymptotically behaves like a linear solution. That the scattering norm \eqref{scattering} is finite is also very important to our analysis. In fact, for $(t,x)\in [0,+\infty)\times \mathbb{R}^3$, $\smash{\|u\|_{\L^5_{t,x}}}<+\infty$ implies that
		$$\||\nabla|^{1/2}u\|_{\L^p_t\L^q_x}\leq C(\|u_0\|_{\dot{\H}^{1/2}_x}, \|u\|_{\L^5_{t,x}})$$
		where $(p,q)$ is a Schr\"odinger admissible pair. See \cite{T-2006} for example.
	\end{rmk}
		
	Based on the results by Kenig and Merle in \cite{K-M-2010}, we prove the following:
	\begin{thm}\label{main-result}
		Let $u(t)$ solve the Cauchy problem \eqref{nls} with initial data $u_0$. Assume that \vspace{-1pt}
		$$\sup_{t\in(0,+\infty)}\|u(t)\|_{\dot{\H}^{1/2}_x} < +\infty,$$
		and $\smash{\|u_0\|_{\dot{\B}^{1/2}_{2,1}}}\leq R_0$. Then there exists a constant $C(\smash{\|u\|_{\L^5_{t,x}}},R_0)>0$ such that
		\begin{equation}\label{main-est}
			\smash{\|u(t,x)\|_{\L^{\infty}_x}\leq C(\|u\|_{\L^5_{t,x}},R_0)t^{-3/2}\|u_0\|_{\L^1_x}}
		\end{equation}
		holds for any $t>0$.
	\end{thm}
	
	Recall the following embedding relations:
	\begin{equation}\label{embedding}
		\dot{\B}^{2}_{1,1}(\mathbb{R}^3)\hookrightarrow \dot{\B}^{1/2}_{2,1}(\mathbb{R}^3)\hookrightarrow \dot{\H}^{1/2}_x(\mathbb{R}^3).
	\end{equation}
	In \cite{D-2023}, Dodson established the global well-posedness and scattering results for radially symmetric initial data in the homogeneous Besov space $\dot{\B}^{2}_{1,1}(\mathbb{R}^3)$:
	\begin{prop}[Theorem 2 in \cite{D-2023}]
		If $u_0\in \smash{\dot{\B}^{2}_{1,1}(\mathbb{R}^3)}$ and $u_0$ is radially symmetric, then \eqref{nls} has a global solution $u(t)$ with initial data $u_0$ and there exists $u^{+}_0\in \smash{\dot{\H}^{1/2}_x(\mathbb{R}^3)}$ such that
		\begin{equation}\label{linear-like-2}
			\lim_{t\to+\infty}\left\|u-\e^{\i t\Delta}u_0^+\right\|_{\dot{\H}^{1/2}_x} = 0.
		\end{equation}
		And if $\|u_0\|_{\dot{\B}^{2}_{1,1}}\leq R_0$, there exists a function $f:[0,\infty)\to [0,\infty)$ such that
		\begin{equation}\label{scattering-2}
			\int_{0}^{+\infty}\int_{\mathbb{R}^3} |u(t,x)|^5\ \mathrm{d}x\mathrm{d}t\leq f(R_0).
		\end{equation}
	\end{prop}
	Based on this result, by the same argument as for Theorem \ref{main-result}, we obtain the following unconditional result:
	\begin{thm}\label{main-result-2}
		Assume that $u_0\in \dot{\B}^{2}_{1,1}(\mathbb{R}^3)\cap \L^1_x(\mathbb{R}^3)$ is radially symmetric and $\smash{\|u_0\|_{\dot{\B}^{2}_{1,1}}}\leq R_0$ and $u(t)$ is the corresponding global solution to the Cauchy problem \eqref{nls} with initial data $u_0$. Then there exists a constant $C(R_0)>0$ such that
		\begin{equation}\label{main-est-2}
			\|u(t,x)\|_{\L^{\infty}_x}\leq C(R_0)t^{-3/2}\|u_0\|_{\L^1_x}
		\end{equation}
		holds for any $t>0$.
	\end{thm}
	\begin{rmk}\label{rmk-on-wp-2}
		Since the global well-posedness and scattering results are established, we do not need to assume the boundedness of the solutions.
		
		As mentioned in Remark \ref{rmk-on-wp-1}, by \eqref{scattering-2} and Strichartz estimates, we can obtain
		$$\||\nabla|^{1/2}u\|_{\L^p_t\L^q_x([0,+\infty)\times \mathbb{R}^3)}\leq C(R_0)$$
		for all Schr\"odinger admissible pairs $(p,q)$.
	\end{rmk}
	
	\subsection{Auxiliary Results}
	We also summarize some well known results in the theory of dispersive equations that are crucial to our analysis.
	
	\begin{prop}[Perturbation theorem c.f. \cite{K-M-2010}]\label{perturbation-thm}
		Let $\I\subset \mathbb{R}$ be a time interval, and let $t_0\in \I$. Let $u$ be defined on $\mathbb{R}^3\times \I$ such that $\smash{\sup_{t\in\I}\|u(t)\|_{\dot{\H}^{1/2}_x}}\leq A$, $\smash{\|u\|_{\L^5_{t,x}(\I)}}\leq M$ and $\smash{\||\nabla|^{1/2}u\|_{\L^5_t\L^{30/11}_x(\I)}}<+\infty$ for some constants $M,A>0$. Assume that $u$ solves
		$$\i\partial_t u + \Delta u - |u|^2u = e,\quad (x,t)\in \mathbb{R}^3\times \I,$$
		and let $v_0\in \dot{\H}^{1/2}_x$ be such that $\|v_0 - u(t_0)\|_{\dot{\H}^{1/2}_x}\leq A'$.
		
		Then there exists $\varepsilon_0 = \varepsilon_0(M,A,A')>0$ such that if $0<\varepsilon\leq \varepsilon_0$ and
		$$\||\nabla|^{1/2}e\|_{\L^{5/3}_t\L^{30/23}_x(\I)}\leq \varepsilon,\quad \|\e^{\i(t-t_0)\Delta}[v_0-u(t_0)]\|_{\L^5_{t,x}(\I)}\leq \varepsilon,$$
		then there exists a unique solution $v(t)$ of \eqref{nls} on $\mathbb{R}^3\times \I$ such that $v\vert_{t=t_0} = v_0$ and
		$$\|v\|_{\L^5_{t,x}(\I)}\leq C(A,A',M),\quad \|v-u\|_{\L^5_{t,x}(\I)}\leq C(A,A',M)(\varepsilon+\varepsilon'),$$
		$$\sup_{t\in \I}\|v(t)-u(t)\|_{\dot{\H}^{1/2}_x} + \||\nabla|^{1/2}(v-u)\|_{\L^5_t\L^{30/11}_x(\I)}\leq C(A,A',M)(A'+\varepsilon+\varepsilon'),$$
		where $\varepsilon' = \varepsilon^{\beta}$ for some $\beta>0$.
	\end{prop}
	
	We also need to use the Strichartz estimates for Schr\"odinger in our analysis:
	\begin{prop}[Strichartz estimates for Schr\"odinger, $d=3$, c.f. \cite{G-V-1989,K-T-1998,Y-1987}]\label{strichartz-est}
		Call a pair $(p,q)$ Schr\"odinger admissible for $d=3$ if $2\leq p,q\leq \infty$ and $\frac{2}{p}+\frac{3}{q} = \frac{3}{2}$. Then for any admissible pair $(p,q)$ and $(\tilde{p},\tilde{q})$, we have the homogeneous Strichartz estimate
		\begin{equation}\label{homogeneous-Strichartz-estimate}
			\|\e^{\i t\Delta}u_0\|_{\L^p_t\L^q_x(\mathbb{R}\times \mathbb{R}^3)}\lesssim_{p,q}\|u_0\|_{\L^2_x(\mathbb{R}^3)},
		\end{equation}
		the dual homogeneous Strichartz estimate
		\begin{equation}\label{dual-homogeneous-Strichartz-estimate}
			\left\|\int_{\mathbb{R}}\e^{-\i s\Delta}F(s)\ \mathrm{d}s\right\|_{\L^2_x(\mathbb{R}^3)}\lesssim_{\tilde{p},\tilde{q}}\|F\|_{\L^{\tilde{p}'}_t\L^{\tilde{q}'}_x(\mathbb{R}\times\mathbb{R}^3)},
		\end{equation}
		and the inhomogeneous Strichartz estimate
		\begin{equation}\label{inhomogeneous-Strichartz-estimate}
			\left\|\int_{s<t}\e^{\i(t-s)\Delta}F(s)\ \mathrm{d}s\right\|_{\L^p_t\L^q_x(\mathbb{R}\times \mathbb{R}^3)}\lesssim_{p,q,\tilde{p},\tilde{q}}\|F\|_{\L^{\tilde{p}'}_t\L^{\tilde{q}'}_x(\mathbb{R}\times\mathbb{R}^3)}.
		\end{equation}
	\end{prop}
	\subsection{Notations}
	We use $C$ to denote the universal constant and $C$ may change line by line. We write $A\lesssim B$ if $A\leq CB$ for some universal constant $C>0$. We write $A\sim B$, if both $A\lesssim B$ and $B\lesssim A$ hold.
	
	Our definition for the Fourier transform is
	$$\hat{f}(\xi) = (2\pi)^{-\frac{d}{2}}\int_{\mathbb{R}^d} f(x) \e^{-\i\xi\cdot x}\ \mathrm{d}x,$$
	so that
	$$f(x) = (2\pi)^{-\frac{d}{2}}\int_{\mathbb{R}^d} \hat{f}(\xi)\e^{\i\xi\cdot x}\ \mathrm{d}\xi\quad \text{and}\quad \int_{\mathbb{R}^d} |\hat{f}(\xi)|^2\ \mathrm{d}\xi = \int_{\mathbb{R}^d} |f(x)|^2\ \mathrm{d}x.$$
	
	We define the \textit{Littlewood-Paley projections} as follows: Let $\varphi$ be a radial bump function supported in the ball $\{|\xi|\leq 2\}$ and $\varphi(\xi) = 1$ on the ball $\{|\xi|\leq 1\}$. For each $N > 0$, we define Fourier multipliers
	$$\begin{aligned}
		\widehat{\mathbf{P}_{\leq N}f}(\xi) &= \varphi\left(\xi/N\right)\hat{f}(\xi),\\
		\widehat{\mathbf{P}_{> N}f}(\xi) &= \left[1-\varphi\left(\xi/N\right)\right]\hat{f}(\xi),\\
		\widehat{\mathbf{P}_{N}f}(\xi) &= \left[\varphi\left(\xi/N\right) - \varphi\left(2\xi/N\right)\right]\hat{f}(\xi).
	\end{aligned}$$
	We denote $\mathbf{P}_{N} f = f_{N}$, $\mathbf{P}_{\leq N} f = f_{\leq N}$ and $\mathbf{P}_{>N} f = f_{>N}$. If a function $f(t,x)$ depends on both time and space, let $f_{N}(t,x)$ denote the Littlewood-Paley projection of $f$ in the spatial variable $x$ only. We will use these multipliers for dyadic $N$.
	
	Fix $s\geq 0$. The homogeneous Sobolev space $\dot{\H}^{s}$ is the completion of the Schwartz class $\mathcal{S}$ with respect to the norm
	\begin{equation}\label{Sobolev-norm}
		\|f\|^2_{\dot{\H}^s} = \||\nabla|^sf\|^2_{\L^2} = \int |\xi|^{2s}|\hat{f}(\xi)|^2\ \mathrm{d}\xi.
	\end{equation}
	
	Fix $1\leq p,q\leq \infty$ and $s\in \mathbb{R}$. The homogeneous Besov space $\smash{\dot{\B}^{s}_{p,q}}$ is the completion of the Schwartz functions $\mathcal{S}$ with respect to the norm
	\begin{equation}\label{Besov-norm}
		\|f\|_{\dot{\B}^{s}_{p,q}} = \left\|N^s\|f_N(x)\|_{\L^p_x(\mathbb{R}^d)}\right\|_{\ell^q_N(2^{\mathbb{Z}})}.
	\end{equation}
	
	As Fourier multipliers, the Littlewood-Paley projections commute with derivative operators and the linear Schr\"odinger propagator. Moreover, given a Schwartz function $f$, let
	$$S(f)(x) := \left(\sum_{N} |\mathbf{P}_N f(x)|^2\right)^{1/2}.$$
	For $1<p<\infty$,
	$$\|S(f)\|_{\L^p(\mathbb{R}^d)}\sim \|f\|_{\L^p(\mathbb{R}^d)}.$$
	More generally,
	$$\left\|\left(\sum_{N} N^{2s}|\mathbf{P}_N f(x)|^2\right)^{1/2}\right\|_{\L^p(\mathbb{R}^d)}\sim \||\nabla|^s f\|_{\L^p(\mathbb{R}^d)}$$
	for $s>-d$ and $1<p<\infty$.
	
	Moreover, for all $N$, the operators $\mathbf{P}_N$ and $\mathbf{P}_{\leq N}$ satisfy
	\begin{equation}\label{bdd-by-maximal-function}
		|\mathbf{P}_N f| + |\mathbf{P}_{\leq N} f| \lesssim \mathbf{M} f
	\end{equation}
	where $\mathbf{M}$ denotes the Hardy–Littlewood maximal function.
	
	Associated to the Littlewood-Paley projections are the \textit{Bernstein inequalities}, which state that
	\begin{align}
		\||\nabla|^s \mathbf{P}_N f\|_{\L^p(\mathbb{R}^d)} &\sim N^s\|\mathbf{P}_N f\|_{\L^p(\mathbb{R}^d)},\label{Bernstein-inequalities-1}\\
		\||\nabla|^s \mathbf{P}_{\leq N} f\|_{\L^p(\mathbb{R}^d)} &\lesssim N^s\|\mathbf{P}_{\leq N} f\|_{\L^p(\mathbb{R}^d)},\label{Bernstein-inequalities-2}\\ 
		\|\mathbf{P}_N f\|_{\L^q(\mathbb{R}^d)} &\lesssim N^{d\left(\frac{1}{p}-\frac{1}{q}\right)}\|\mathbf{P}_N f\|_{\L^p(\mathbb{R}^d)},\label{Bernstein-inequalities-3}\\
		\|\mathbf{P}_{\leq N}f\|_{\L^q(\mathbb{R}^d)} &\lesssim N^{d\left(\frac{1}{p}-\frac{1}{q}\right)}\|\mathbf{P}_{\leq N} f\|_{\L^p(\mathbb{R}^d)},\label{Bernstein-inequalities-4}
	\end{align}
	for all $s\geq 0$ and $1\leq p\leq q\leq \infty$. For more details on Littlewood--Paley theory, one can refer to \cite{G-2014}.
	
	We write $a \vee b := \max(a,b)$ and $a \wedge b := \min(a,b)$.
	
	For the embedding relation between two function spaces $X$ and $Y$, we use the notation $X \hookrightarrow Y$ to indicate that $X$ is continuously embedded in $Y$, namely there exists a constant $C>0$ such that $\|f\|_Y \leq C \|f\|_X$ for all $f \in X$.
	
	We use $\mathcal{O}(X)$ to denote a term that is \textit{schematically} like $X$. That is a finite linear combination of terms that look like $X$ but potentially with some terms replaced by their absolute value, complex conjugate or a Littlewood-Paley projection. For example,
	$$|v+w|^2(v+w) = \sum_{j=0}^{3}\mathcal{O}(v^{j}w^{3-j})\quad \text{and}\quad f_{\leq N/8}\cdot g_{N_1}\cdot h = \mathcal{O}(fg_{N_1}h).$$
	
	We use $\L^p_t\L^q_x(T\times X)$ to denote the mixed Lebesgue spacetime norm
	$$\|f\|_{\L^p_t\L^q_x(T\times X)} = \left\|\|f(t,x)\|_{\L^q_x(X)}\right\|_{\L^p_t(T)} = \left[\int_T\left(\int_X |f(t,x)|^q\ \mathrm{d}x\right)^{p/q}\ \mathrm{d}t\right]^{1/p}.$$
	When $p=q$, we denote $\L^p_{t,x} = \L^p_t\L^q_x$. When $X=\mathbb{R}^d$, we abbreviate $\L^p_t\L^q_x(T) = \L^p_t\L^q_x(T\times \mathbb{R}^d)$. Throughout this note, unless explicitly stated, $\L^p_t\L^q_x$ stands for $\L^p_t\L^q_x([0,+\infty)\times \mathbb{R}^3)$.
	\subsection{Acknowledgment}
	This research was partially supported by NSFC, Grant No. 12471232, 12288201, CAS Project for Young Scientists in Basic Research, Grant No. YSBR-031. The author would like to thank his advisor Professor Chenjie Fan for insightful discussions and comments on the proofs. This work would not have taken its present form without his guidance. The author is also grateful to the anonymous reviewer for the valuable suggestions and insightful feedback that improve the clarity and quality of this work.
	
	\section{Preliminaries}
	It will be important to employ a paraproduct decomposition to the cubic nonlinearity of \eqref{nls}. We use the following paraproduct decomposition from [\cite{K-2025}, Lemma 4.2]:
	\begin{lem}[Paraproduct decomposition]\label{paraproduct_decomposition}
		Suppose that $f^1,f^2,f^3\in \L^2$. Then $\mathbf{P}_N(f^1f^2f^3)$ can be expressed as
		$$\begin{aligned}
			\mathbf{P}_N(f^1f^2f^3) &= \mathbf{P}_N\left(f^1_{\geq N/8}f^2f^3 + f^1_{<N/8}f^2_{\geq N/8}f^3 + f^1_{<N/8}f^2_{<N/8}f^3_{\geq N/8}\right)\\
			&= \mathbf{P}_N\sum_{N_1\gtrsim N}\sum_{(i,j,k)}\mathcal{O}\left(f^{i}_{N_1}f^{j}f^{k}\right),
		\end{aligned}$$
		where $(i,j,k)$ is summed over all permutations of $(1,2,3)$.
		
		In particular,
		$$\left|\mathbf{P}_N(f^1f^2f^3)\right|\lesssim \sum_{(i,j,k)}\sum_{N_1\gtrsim N}\mathbf{M}\left[f^{i}_{N_1}\cdot \mathbf{M} f^{j} \cdot \mathbf{M} f^{k}\right],$$
		uniformly for $N\in 2^{\mathbb{Z}}$ where $\mathbf{M}$ is the Hardy-Littlewood maximal function.
	\end{lem}
	
	The following Propositions \ref{Besov-spacetime-bounds} and \ref{B-stability} are analogues in our setting of Propositions 4.3 and 4.4 in \cite{K-2025}, respectively.
	\begin{prop}[Besov spacetime bounds]\label{Besov-spacetime-bounds}
		Let $u(t)$ be the global solution to \eqref{nls} with initial data $u_0\in \smash{\dot{\B}^{1/2}_{2,1}(\mathbb{R}^3)}$. Let $(p,q)$ be a Schr\"odinger-admissible pair for $d=3$ i.e. $(p,q)$ satisfies $\smash{\frac{2}{p}+\frac{3}{q}=\frac{3}{2}}$. Suppose that
		$$\sup_{t\in (0,+\infty)}\|u(t)\|_{\dot{\H}^{1/2}_x} < +\infty.$$
		Then we have
		\begin{equation}\label{Besov-est}
			\sum_{N\in 2^{\mathbb{Z}}}N^{1/2}\left\|u_N\right\|_{\L^p_t\L^q_x}\leq C(\|u\|_{\L^5_{t,x}})\|u_0\|_{\dot{\B}^{1/2}_{2,1}}.
		\end{equation}
	\end{prop}
	\begin{proof}
		We first establish the estimate for the admissible pair $(10/3,10/3)$. We introduce a small parameter $\eta>0$ which is fixed and chosen later. By Proposition \ref{wellposed-result}, we partition $[0,+\infty)$ into $J=J(\eta,\|u\|_{\L^5_{t,x}})$ many intervals $\I_{j} = [t_j,t_{j+1})$ on which
		\begin{equation}\label{small-L5}
			\left\|u\right\|_{\L^5_{t,x}(\I_j\times \mathbb{R}^3)}<\eta.
		\end{equation}
		By the Duhamel formula and Strichartz estimates, for each spacetime slab $\I_j\times\mathbb{R}^3$, we have
		$$\sum_{N\in 2^{\mathbb{Z}}}N^{1/2}\left\|u_N\right\|_{\L^{10/3}_{t,x}(\I_j)} \lesssim \|u_0\|_{\dot{\B}^{1/2}_{2,1}} + \sum_{N\in 2^{\mathbb{Z}}} N^{1/2}\left\|\mathbf{P}_{N}\left(|u|^2u\right)\right\|_{\L^{10/7}_{t,x}(\I_j)}.$$
		We estimate the non-linear term using Lemma \ref{paraproduct_decomposition}. The paraproduct decomposition gives
		$$|\mathbf{P}_{N}(|u|^2u)|\lesssim \sum_{N_1 \gtrsim N}\mathbf{M}\left[u_{N_1}(\mathbf{M} u)^2\right].$$
		Hence H\"older's inequality and the boundedness of the Hardy-Littlewood maximal operator imply
		$$\sum_{N}N^{1/2}\|\mathbf{P}_{N}(|u|^2u)\|_{\L^{10/7}_{t,x}(\I_j)}\lesssim \sum_{N}\sum_{N_1 \gtrsim N}N^{1/2}\|u\|^2_{\L^5_{t,x}(\I_j)}\|u_{N_1}\|_{\L^{10/3}_{t,x}(\I_j)}.$$
		Summing over $N$ first, we obtain
		$$\sum_{N\in 2^{\mathbb{Z}}}N^{1/2}\|\mathbf{P}_{N}(|u|^2u)\|_{\L^{10/7}_{t,x}(\I_j)}\lesssim \eta^2\sum_{N\in 2^{\mathbb{Z}}}N^{1/2}\left\|u_N\right\|_{\L^{10/3}_{t,x}(\I_j)}.$$
		It follows that
		\begin{equation}\label{Besov-est-2}
			\sum_{N\in 2^{\mathbb{Z}}}N^{1/2}\left\|u_N\right\|_{\L^{10/3}_{t,x}(\I_j)} \lesssim \|u_0\|_{\dot{\B}^{1/2}_{2,1}} + \eta^2\sum_{N_1\in 2^{\mathbb{Z}}}N_1^{1/2}\left\|u_{N_1}\right\|_{\L^{10/3}_{t,x}(\I_j)}.
		\end{equation}
		Choosing $\eta$ sufficiently small depending on some absolute constants, we obtain
		$$\sum_{N\in 2^{\mathbb{Z}}}N^{1/2}\left\|u_N\right\|_{\L^{10/3}_{t,x}(\I_j)}\lesssim \|u_0\|_{\dot{\B}^{1/2}_{2,1}}.$$
		Summing over $j=1,\dots,J(\smash{\|u\|_{\L^5_{t,x}}})$, we obtain
		$$\sum_{N\in 2^{\mathbb{Z}}}N^{1/2}\left\|u_N\right\|_{\L^{10/3}_{t,x}}\leq C(\|u\|_{\L^5_{t,x}})\|u_0\|_{\dot{\B}^{1/2}_{2,1}}.$$
		Then by Strichartz estimates, for any Schr\"odinger admissible pair $(p,q)$, we have
		$$\sum_{N\in 2^{\mathbb{Z}}}N^{1/2}\left\|u_N\right\|_{\L^p_t\L^q_x}\lesssim \|u_0\|_{\dot{\B}^{1/2}_{2,1}} + \|u\|^2_{\L^5_{t,x}}\sum_{N_1\in 2^{\mathbb{Z}}} N^{1/2}_1\left\|u_{N_1}\right\|_{\L^{10/3}_{t,x}}\leq C(\|u\|_{\L^5_{t,x}})\|u_0\|_{\dot{\B}^{1/2}_{2,1}}.$$
		The proof is complete.
	\end{proof}
	\begin{rmk}
		Besov-space techniques have also been developed in \cite{N-W-2019} by Nakamura and Wada. They established Strichartz-type estimates in mixed Besov spaces and applied them to obtain global small-data solutions for certain critical nonlinear Schr\"odinger equations.
	\end{rmk}
	
	\begin{prop}[$\dot{\B}^{1/2}_{2,1}$ stability]\label{B-stability}
		Let $(p,q)$ be a Schr\"odinger admissible pair for $d=3$. Suppose that $u(t)$ and $v(t)$ are the solutions to \eqref{nls} satisfying
		$$\sup_{t\in(0,+\infty)}\|u(t)\|_{\dot{\H}^{1/2}_x},\ \sup_{t\in (0,+\infty)}\|v(t)\|_{\dot{\H}^{1/2}_x} < +\infty$$
		with corresponding initial data $u_0,v_0$ respectively obeying the bound
		$$\|u_0\|_{\dot{\B}^{1/2}_{2,1}},\ \|v_0\|_{\dot{\B}^{1/2}_{2,1}}\leq R_0.$$
		Then
		\begin{equation}\label{Besov-stability}
			\sum_{N\in 2^{\mathbb{Z}}}N^{1/2}\|(u-v)_N\|_{\L^p_t\L^q_x}\leq C(S,R_0)\|u_0-v_0\|_{\dot{\B}^{1/2}_{2,1}}
		\end{equation}
		where $S = \|u\|_{\L^5_{t,x}} \vee \|v\|_{\L^5_{t,x}}$.
	\end{prop}
	\begin{proof}
		Set $w = u-v$ and $U = |u| + |v|$. We first consider the admissible pair $(10/3,10/3)$. Let
		$$S := \|u\|_{\L^5_{t,x}} \vee \|v\|_{\L^5_{t,x}}$$
		Fix $\eta>0$. Decompose $[0,+\infty)$ into $J=J(\eta,S)$ many intervals $\I_j$ such that
		\begin{equation}\label{B-stability-est-1}
			\left\||u|\vee |v|\right\|_{\L^5_{t,x}(\I_j)}\leq \|u\|_{\L^5_{t,x}(\I_j)} + \|v\|_{\L^5_{t,x}(\I_j)}<\eta.
		\end{equation}
		$w$ satisfies
		$$\i \partial_t w + \Delta w = |u|^2u - |v|^2v = \mathcal{O}(wU^2).$$
		Strichartz estimates yield
		\begin{equation}\label{diff_10/3_est}
			\sum_{N\in 2^{\mathbb{Z}}} N^{1/2}\|w_N\|_{\L^{10/3}_{t,x}(\I_j)}\lesssim \|u_0 - v_0\|_{\dot{\B}^{1/2}_{2,1}} + \sum_{N\in 2^{\mathbb{Z}}} N^{1/2}\|\mathbf{P}_{N}\mathcal{O}(w U^2)\|_{\L^{10/7}_{t,x}(\I_j)}.
		\end{equation}
		The paraproduct terms fall into two cases.
		
		\noindent{\bf Case 1: the frequency is dominated by $w$.} By H\"older's inequality and the boundedness of the Hardy-Littlewood maximal operator, we have
		$$\begin{aligned}
			\sum_{N}\sum_{N_1 \gtrsim N} N^{1/2}\|w_{N_1}(\mathbf{M}U)^2\|_{\L^{10/7}_{t,x}(\I_j)} &\lesssim \|U\|^2_{\L^5_{t,x}(\I_j)}\sum_{N_1\in 2^{\mathbb{Z}}} N^{1/2}_1\|w_{N_1}\|_{\L^{10/3}_{t,x}(\I_j)}\\
			&\lesssim \eta^2\sum_{N\in 2^{\mathbb{Z}}} N^{1/2}\|w_{N}\|_{\L^{10/3}_{t,x}(\I_j)}.
		\end{aligned}$$
		
		\noindent{\bf Case 2: the frequency is dominated by $u$ or $v$.} We have
		$$\begin{aligned}
			\sum_{N_1\in 2^{\mathbb{Z}}} N^{1/2}_1\|(u_{N_1} + v_{N_1})(\mathbf{M}U)(\mathbf{M}w)\|_{\L^{10/7}_{t,x}(\I_j)}\lesssim &\sum_{N_1\in 2^{\mathbb{Z}}} N^{1/2}_1\left(\|u_{N_1}\|_{\L^5_t\L^{30/11}_x(\I_j)} + \|v_{N_1}\|_{\L^5_t\L^{30/11}_x(\I_j)}\right)\\
			&\cdot \|U\|_{\L^5_{t,x}(\I_j)}\|w\|_{\L^{10/3}_t\L^{15/2}_x(\I_j)}.
		\end{aligned}$$
		By the Bernstein inequalities,
		$$\|w\|_{\L^{10/3}_t\L^{15/2}_x}\lesssim \sum_{N\in 2^{\mathbb{Z}}} N^{1/2}\|w_{N}\|_{\L^{10/3}_{t,x}(\I_j)}.$$
		Proposition \ref{Besov-spacetime-bounds} yields
		$$\sum_{N_1\in 2^{\mathbb{Z}}} N^{1/2}_1\left(\|u_{N_1}\|_{\L^5_t\L^{30/11}_x} + \|v_{N_1}\|_{\L^5_t\L^{30/11}_x}\right)\leq C(S,R_0).$$
		Combining the two cases, \eqref{diff_10/3_est} yields
		$$\sum_{N\in 2^{\mathbb{Z}}} N^{1/2}\|w_N\|_{\L^{10/3}_{t,x}(\I_j)}\lesssim \|u_0 - v_0\|_{\dot{\B}^{1/2}_{2,1}} + C(S,R_0)\eta \sum_{N\in 2^{\mathbb{Z}}} N^{1/2}\|w_N\|_{\L^{10/3}_{t,x}(\I_j)}.$$
		Choosing $\eta = \eta(S,R_0)$ sufficiently small, we obtain
		$$\sum_{N\in 2^{\mathbb{Z}}} N^{1/2}\|w_N\|_{\L^{10/3}_{t,x}(\I_j)} \leq C(S,R_0)\|u_0 - v_0\|_{\dot{\B}^{1/2}_{2,1}}.$$
		Summing over $j=1,\dots,J(S,R_0)$, we obtain
		$$\sum_{N\in 2^{\mathbb{Z}}}N^{1/2}\left\|w_N\right\|_{\L^{10/3}_{t,x}}\leq C(S,R_0)\|u_0 - v_0\|_{\dot{\B}^{1/2}_{2,1}}.$$
		Then by Strichartz estimates, for any Schr\"odinger admissible pair $(p,q)$, we have
		$$\sum_{N\in 2^{\mathbb{Z}}}N^{1/2}\left\|w_N\right\|_{\L^p_t\L^q_x}\leq C(S,R_0)\|u_0 - v_0\|_{\dot{\B}^{1/2}_{2,1}}.$$
		The proof is complete.
	\end{proof}
	
	\section{Proof of the main theorems}
	Let $u(t)$ be the solution to \eqref{nls} with initial data $u_0$ satisfying $\|u_0\|_{\dot{\B}^{1/2}_{2,1}}\leq R_0$ and $$\sup_{t\in(0,+\infty)}\|u(t)\|_{\dot{\H}^{1/2}_x} < +\infty.$$
	Define
	\begin{equation}\label{A}
		A(t) := \sup_{0<s\leq t}s^{3/2}\|u(s)\|_{\L^{\infty}_x}.
	\end{equation}
	Note that $A(t)$ is monotone increasing. Our aim is to prove there is some constant $C(\smash{\|u\|_{\L^5_{t,x}}},R_0)$ such that
	\begin{equation}\label{aim-est}
		A(t)\leq C(\|u\|_{\L^5_{t,x}},R_0)\|u_0\|_{\L^1_x},\ \text{for any $t>0$}.
	\end{equation}
	By the density of Schwartz functions in $\smash{\dot{\B}^{1/2}_{2,1}\cap \L^1}$, it suffices to consider Schwartz solutions. Then for any given $l$, one can find $C_{l}$ such that
	\begin{equation}\label{est-hold-some-t}
		A(t)\leq C_{l}\|u_0\|_{\L^1_x},\quad 0< t\leq l.
	\end{equation}
	Also the solution is continuous in time in $\L^{\infty}$.
	
	We now prove Theorem \ref{main-result} by a bootstrap argument.
	
	\subsection{Base case.} Theorem \ref{main-result} holds trivially for $\smash{\|u_0\|_{\dot{\B}^{1/2}_{2,1}}}=0$. We take the small-data case as the base case for the induction. We aim to prove:
	\begin{thm}\label{base-thm}
		Let $u(t)$ be the solution to \eqref{nls} with initial data $u_0$. Assume $\sup_{t\in(0,+\infty)}\|u(t)\|_{\dot{\H}^{1/2}_x}<+\infty$. There exists $0<\varepsilon<1$ depending on $\|u\|_{\L^5_{t,x}}$ and some absolute constants such that if $\smash{\|u_0\|_{\dot{\B}^{1/2}_{2,1}}}\leq \varepsilon$, then there exists a constant $C(\smash{\|u\|_{\L^5_{t,x}}})$ such that
		$$\|u\|_{\L^{\infty}_x}\leq C(\smash{\|u\|_{\L^5_{t,x}}})t^{-3/2}\|u_0\|_{\L^1_x},\quad \text{for any $t>0$}.$$
	\end{thm}
	
	It suffices to prove the following bootstrap lemma:
	\begin{lem}\label{base_case_bootstrap_lemma}
		Let $u(t)$ be the solution to \eqref{nls} with initial data $u_0$. Assume $\sup_{t\in(0,+\infty)}\|u(t)\|_{\dot{\H}^{1/2}_x}<+\infty$. There exists $0<\varepsilon<1$ depending on $\|u\|_{\L^5_{t,x}}$ and some absolute constants such that, if $\smash{\|u_0\|_{\dot{\B}^{1/2}_{2,1}}}\leq \varepsilon$, then there exists $C_0(\|u\|_{\L^5_{t,x}})>0$ such that for any $L>0$, if one has $A(L)\leq C_0\|u_0\|_{\L^1_x}$, then one has $A(L)\leq \tfrac{C_0}{2}\|u_0\|_{\L^1_x}$.
	\end{lem}
	\begin{rmk}\label{close_bootstrap}
		Lemma \ref{base_case_bootstrap_lemma} indeed closes the bootstrap. Since the map $t \mapsto t^{3/2}\|u(t)\|_{\L^{\infty}_x}$ is continuous for $t>0$, $A(t)$ is continuous. Moreover, \eqref{est-hold-some-t} implies the estimate holds for sufficiently small positive times.
		
		Define
		$$\Omega := \left\{L > 0:\ A(L) \leq C_0\|u_0\|_{\L^1_x}\right\}.$$
		The set $\Omega$ is nonempty. If $L\in \Omega$, Lemma \ref{base_case_bootstrap_lemma} improves the estimate to
		$$A(L) \leq \tfrac{C_0}{2}\|u_0\|_{\L^1_x}.$$
		By continuity, the original estimate with constant $C_0$ remains valid for times in a neighborhood of $L$, thus $\Omega$ is open. The set is also closed by continuity of $A(t)$. Since $(0,\infty)$ is connected, $\Omega = (0,\infty)$. Consequently,
		$$A(t)\leq C_0\|u_0\|_{\L^1_x}$$
		for every $t > 0$.
	\end{rmk}
	\begin{proof}[Proof of Lemma {\rm \ref{base_case_bootstrap_lemma}}]
		For any $0<t\leq L$, by Duhamel's formula, write
		\begin{equation}\label{Duhamel}
			u(t,x) = \e^{\i t\Delta}u_0 - \i\int_{0}^{t}\e^{\i(t-s)\Delta}(|u|^2u)(s)\ \mathrm{d}s = u_{\text{lin}} + u_{\text{nl}}.
		\end{equation}
		We split $u_{\text{nl}}$ into
		\begin{equation}\label{u_nl}
			u_{\text{nl}} = F_1(t) + F_2(t),
		\end{equation}
		where
		$$\begin{aligned}
			F_1(t) &= -\i\int_{0}^{t/2}\e^{\i(t-s)\Delta}(|u|^2u)(s)\ \mathrm{d}s,\\
			F_2(t) &= -\i\int_{t/2}^{t}\e^{\i(t-s)\Delta}(|u|^2u)(s)\ \mathrm{d}s.
		\end{aligned}$$
		
		\noindent\textbf{Linear term} $u_{\text{lin}}${\bf.} By dispersive estimate, we have
		\begin{equation}\label{u_l-est}
			\|u_{\text{lin}}(t)\|_{\L^{\infty}_x}\lesssim t^{-3/2}\|u_0\|_{\L^1_x}.
		\end{equation}
		
		\noindent\textbf{Early-time term $F_1(t)$.} Note that $\dot{\H}^{1/2}_x(\mathbb{R}^3)$ is embedded into $\L^3_x(\mathbb{R}^3)$. It turns out that $u_0\in \L^3_x$ and $u(t)\in \L^{\infty}_t\L^3_x$. Since $u_0\in \L^1_x$ and the mass $m=\int |u(x,t)|^2\ \mathrm{d}x$ is conserved, we know $u_0\in \L^2_x$ and by H\"older's inequality, we have
		\begin{equation}\label{u_0-L2-set}
			\|u\|^2_{\L^2_x}=\|u_0\|^2_{\L^2_x}\leq \|u_0\|^{1/2}_{\L^1_x}\|u_0\|^{3/2}_{\L^3_x}\lesssim \|u_0\|^{1/2}_{\L^1_x}\|u_0\|_{\dot{\H}^{1/2}_x}^{3/2}\lesssim \|u_0\|^{1/2}_{\L^1_x}.
		\end{equation}
		
		By dispersive estimate, we estimate $F_1$ as
		\begin{equation}\label{F1-est-1}
			\|F_1(t)\|_{\L^{\infty}_x}\lesssim t^{-3/2}\|u(s)\|^3_{\L^3_s\L^3_x((0,+\infty))}.
		\end{equation}
		We introduce a parameter $M$ and estimate \eqref{F1-est-1} as
		\begin{equation}\label{F1-est-2}
			\|u(s)\|^3_{\L^3_s\L^3_x((0,+\infty))} \leq \|u(s)\|^3_{\L^3_s\L^3_x((0,M))} + \int_{M}^{+\infty} \|u(s)\|_{\L^\infty_x}\|u\|^2_{\L^2_x}\ \mathrm{d}s.
		\end{equation}
		By Remark \ref{rmk-on-wp-1}, $\smash{\sup_{t\in[0,+\infty)}\|u(t)\|_{\dot{\H}^{1/2}_x}}\leq C(\smash{\|u\|_{\L^5_{t,x}}})$. By Sobolev embedding,
		\begin{equation}\label{u-L3}
			\|u(s)\|^3_{\L^3_s\L^3_x((0,M))} \lesssim M\|u(s)\|^3_{\L^{\infty}_s\dot{\H}^{1/2}_x}\leq C(\|u\|_{\L^5_{t,x}})M.
		\end{equation}
		By \eqref{u_0-L2-set} and bootstrap assumption, we obtain
		\begin{equation}\label{F1-est-3}
			\int_{M}^{+\infty} \|u(s)\|_{\L^\infty_x}\|u\|^2_{\L^2_x}\ \mathrm{d}s\lesssim C_0\|u_0\|^{3/2}_{\L^1_x}\int_{M}^{+\infty}s^{-3/2}\ \mathrm{d}s \lesssim C_0M^{-1/2}\|u_0\|^{3/2}_{\L^1_x}.
		\end{equation}
		By taking $M = C\cdot \|u_0\|_{\L^1_x}$, where $C$ is large enough depending only on some universal constants, we obtain
		\begin{equation}\label{F1-final-est-2}
			\|F_1(t)\|_{\L^{\infty}_x}\leq \left(C(\|u\|_{\L^5_{t,x}})+\tfrac{1}{10}C_0\right)t^{-3/2}\|u_0\|_{\L^1_x}
		\end{equation}
		
		\noindent\textbf{Late-time term $F_2(t)$.} By paraproduct decomposition, we estimate $F_2$ as
		\begin{equation}\label{F2-est-1}
			\|F_2(t)\|_{\L^{\infty}_x}\leq \sum_{N,N_1:N_1\gtrsim N}\sum_{N_2}\left\|\int_{t/2}^{t}\mathbf{P}_N\left[\e^{\i(t-s)\Delta}\mathcal{O}\left(u_{N_1}u_{N_2}u\right)\right]\ \mathrm{d}s\right\|_{\L^{\infty}_x}.
		\end{equation}
		Fix $N,N_1,N_2$ and introduce $B>0$ which will be chosen later.
		\begin{equation}\label{F2-term-1}
			\left\|\int_{t/2}^{t}\mathbf{P}_N\left[\e^{\i(t-s)\Delta}\mathcal{O}\left(u_{N_1}u_{N_2}u\right)\right]\ \mathrm{d}s\right\|_{\L^{\infty}_x}
		\end{equation}
		is bounded by the sum of
		\begin{equation}\label{F2-term-2}
			\left\|\int_{t/2}^{(t-B)\vee t/2}\mathbf{P}_N\left[\e^{\i(t-s)\Delta}\mathcal{O}\left(u_{N_1}u_{N_2}u\right)\right]\ \mathrm{d}s\right\|_{\L^{\infty}_x},
		\end{equation}
		\begin{equation}\label{F2-term-3}
			\left\|\int_{(t-B)\vee t/2}^{t}\mathbf{P}_N\left[\e^{\i(t-s)\Delta}\mathcal{O}\left(u_{N_1}u_{N_2}u\right)\right]\ \mathrm{d}s\right\|_{\L^{\infty}_x}.
		\end{equation}
		{\bf Away from $s=t$.} By dispersive estimate, Cauchy-Schwarz and bootstrap assumption,
		\begin{align}
			\eqref{F2-term-2} &\lesssim\int_{t/2}^{(t-B)\vee t/2}|t-s|^{-3/2}\left\|u_{N_1}(s)u_{N_2}(s)\right\|_{\L^1_x}\|u(s)\|_{\L^{\infty}_x}\ \mathrm{d}s\notag\\
			&\lesssim C_0t^{-3/2}B^{-1/2}\left\|u_{N_1}\right\|_{\L^{\infty}_t\L^2_x}\left\|u_{N_2}\right\|_{\L^{\infty}_t\L^2_x}\|u_0\|_{\L^1_x}.\label{F2-term-2-est}
		\end{align}
		{\bf Near $s=t$.} By Bernstein's inequality and bootstrap assumption,
		\begin{align}
			\eqref{F2-term-3}&\lesssim\int_{(t-B)\vee t/2}^{t}N^{3/2}\left\|u_{N_1}(s)u_{N_2}(s)\right\|_{\L^2_x}\|u\|_{\L^{\infty}_x}\ \mathrm{d}s\notag\\
			&\lesssim C_0t^{-3/2}BN^{3/2}\left\|u_{N_1}\right\|_{\L^{\infty}_t\L^2_x}\left\|u_{N_2}\right\|_{\L^{\infty}_{t,x}}\|u_0\|_{\L^1_x}\notag\\
			&\lesssim C_0t^{-3/2}BN^{3/2}N_2^{3/2}\left\|u_{N_1}\right\|_{\L^{\infty}_t\L^2_x}\left\|u_{N_2}\right\|_{\L^{\infty}_t\L^2_x}\|u_0\|_{\L^1_x}.\label{F2-term-3-est}
		\end{align}
		Choosing $B=N^{-1}N^{-1}_2$, we obtain
		\begin{equation}\label{F2-term-1-est}
			\eqref{F2-term-1}\lesssim C_0t^{-3/2}\|u_0\|_{\L^1_x}\left(N^{1/2}\left\|u_{N_1}\right\|_{\L^{\infty}_t\L^2_x}\right)\left(N_2^{1/2}\left\|u_{N_2}\right\|_{\L^{\infty}_t\L^2_x}\right).
		\end{equation}
		Summing over $N$ with $N\lesssim N_1$ firstly, then over $N_1,N_2$ and using \eqref{Besov-est}, we obtain
		\begin{equation}\label{F2-final-est}
			\|F_2(t)\|_{\L^{\infty}_x}\leq C_0C(\smash{\|u\|_{\L^5_{t,x}}})\varepsilon^2t^{-3/2}\|u_0\|_{\L^1_x}.
		\end{equation}
		Choosing $\varepsilon$ small enough depending on $\smash{\|u\|_{\L^5_{t,x}}}$ and some absolute constants, we can ensure
		\begin{equation}\label{F2-final-est-2}
			\|F_2(t)\|_{\L^{\infty}_x}\leq \tfrac{1}{10}C_0t^{-3/2}\|u_0\|_{\L^1_x}.
		\end{equation}
		
		By \eqref{u_l-est}, \eqref{F1-final-est-2} and \eqref{F2-final-est-2}, we can conclude that
		\begin{equation}\label{u-L^infty}
			\|u(t)\|_{\L^{\infty}_x}\leq \left[C(\|u\|_{\L^5_{t,x}}) + \tfrac{1}{5}C_0\right]t^{-3/2}\|u_0\|_{\L^1_x}.
		\end{equation}
		Choosing $C_0 = 10C(\|u\|_{\L^5_{t,x}})$ we can close the bootstrap. The proof is complete.
	\end{proof}	
	
	\subsection{Induction}
	Let $u(t)$ be the solution to \eqref{nls} with initial data $u_0$. Assume inductively that $\smash{\sup_{t\in(0,+\infty)}\|u(t)\|_{\dot{\H}^{1/2}_x}}<+\infty$, and there exists some $R_0\geq 0$ such that if $\smash{\|u_0\|_{\dot{\B}^{1/2}_{2,1}}}\leq R_0$, then
	\begin{equation}\label{goal-est}
		\|u(t)\|_{\L^{\infty}_x}\leq C(\smash{\|u\|_{\L^5_{t,x}}},R_0)t^{-3/2}\|u_0\|_{\L^1_x}
	\end{equation}
	holds for any $t> 0$.
	
	We prove the induction step by increasing the admissible Besov radius from $R_0$ to $R_0+1$. The constant in \eqref{goal-est} is allowed to change in this step.
	
	As in \cite{K-2025}, we finish the induction by showing this increment. The induction argument essentially follows that in \cite{K-2025} with modifications as mentioned in the introduction.
	\begin{proof}[Proof of Theorem {\upshape \ref{main-result}}]
		Fix $\varepsilon>0$ sufficiently small to be chosen later. Suppose that if $v(t)$ is the solution to \eqref{nls} with initial data $v_0$ with $\smash{\sup_{t\in(0,+\infty)}\|v(t)\|_{\dot{\H}^{1/2}_x}} < +\infty$ and $\smash{\|v_0\|_{\dot{\B}^{1/2}_{2,1}}}\leq R_0 + k\varepsilon$, then it holds
		\begin{equation}\label{assumption-est}
			\|v(t)\|_{\L^{\infty}_x}\leq C(\smash{\|v\|_{\L^5_{t,x}}},R_0,k,\varepsilon)t^{-3/2}\|v_0\|_{\L^1_x}.
		\end{equation}
		We aim to show that for all solutions $u(t)$ to \eqref{nls} with initial data $u_0$, if $\smash{\sup_{t\in(0,+\infty)}\|u(t)\|_{\dot{\H}^{1/2}_x}} < +\infty$ and $\|u_0\|_{\dot{\B}^{1/2}_{2,1}}\leq R_0 + (k+1)\varepsilon$, then it holds
		\begin{equation}\label{induction-est-to-prove}
			\|u(t)\|_{\L^{\infty}_x}\leq C(\smash{\|u\|_{\L^5_{t,x}}},R_0,k,\varepsilon)t^{-3/2}\|u_0\|_{\L^1_x},
		\end{equation}
		perhaps with a new constant $C(\|u\|_{\L^5_{t,x}},R_0,k,\varepsilon)$. Once $\varepsilon=\varepsilon(\|u\|_{\L^5_{t,x}},R_0)$ is fixed sufficiently small, finitely many repetitions of this step enlarge the admissible radius from $R_0$ to $R_0+1$.
		
		Let $u(t)$ be the solution to \eqref{nls} with initial data $u_0$ with $\smash{\sup_{t\in(0,+\infty)}\|u(t)\|_{\dot{\H}^{1/2}_x}}<+\infty$ and $\|u_0\|_{\dot{\B}^{1/2}_{2,1}}\leq R_0 + (k+1)\varepsilon$. We aim to prove \eqref{induction-est-to-prove} holds for any $t>0$.
		
		We decompose $u_0$ as
		\begin{equation}\label{decomposition}
			u_0 = \frac{R_0 + k\varepsilon}{R_0+(k+1)\varepsilon}u_0 + \frac{\varepsilon}{R_0+(k+1)\varepsilon}u_0 =: v_0+w_0.
		\end{equation}
		Then $\|v_0\|_{\dot{\B}^{1/2}_{2,1}}\leq R_0 + k\varepsilon$, $\|w_0\|_{\dot{\B}^{1/2}_{2,1}}\leq \varepsilon$ and $\|v_0\|_{\L^1_x},\|w_0\|_{\L^1_x}\leq \|u_0\|_{\L^1_x}$.
		
		By Sobolev embedding and Strichartz estimates \eqref{homogeneous-Strichartz-estimate},
		\begin{equation}\label{v-exists-condition-1}
			\|\e^{\i t\Delta}[u_0-v_0]\|_{\L^5_{t,x}([0,+\infty))}\lesssim \|\e^{\i t\Delta}|\nabla|^{1/2}[u_0-v_0]\|_{\L^5_t\L^{30/11}_x([0,+\infty))}\lesssim \|u_0-v_0\|_{\dot{\H}^{1/2}_x}\lesssim \varepsilon.
		\end{equation}
		We introduce a small parameter $\eta>0$ which is fixed and chosen later and decompose $[0,+\infty)$ into $J=J(\eta,\smash{\|u\|_{\L^5_{t,x}}})$ many intervals $\I_{j} = [t_j,t_{j+1})$ on which
		$$\left\|u\right\|_{\L^5_{t,x}(\I_j\times \mathbb{R}^3)}<\eta.$$
		By Strichartz estimates, fractional Leibniz (see \cite{L-2019} for example) and H\"older's inequality,
		\begin{align}
			\||\nabla|^{1/2}u\|_{\L^5_t\L^{30/11}_x(\I_j)}&\lesssim (R_0+1) + \|u\|^2_{\L^5_{t,x}(\I_j)}\||\nabla|^{1/2}u\|_{\L^5_t\L^{30/11}_x(\I_j)}\notag\\
			&\lesssim (R_0+1) + \eta^2\||\nabla|^{1/2}u\|_{\L^5_t\L^{30/11}_x(\I_j)}. \label{u-norm-condition-1}
		\end{align}
		Choosing $\eta$ small enough depending on $R_0$ and absolute constants in \eqref{u-norm-condition-1}, a standard bootstrap argument yields
		$$\||\nabla|^{1/2}u\|_{\L^5_t\L^{30/11}_x(\I_j)}\leq C(R_0).$$
		Summing over $j=1,\dots,J(\|u\|_{\L^5_{t,x}},R_0)$, we obtain
		\begin{equation}\label{v-exists-condition-2}
			\||\nabla|^{1/2}u\|_{\L^5_t\L^{30/11}_x([0,+\infty))}\leq C(\smash{\|u\|_{\L^5_{t,x}}},R_0).
		\end{equation}
		Also we have
		\begin{equation}\label{v-exists-condition-3}
			\|u_0-v_0\|_{\dot{\H}^{1/2}_x}\leq \varepsilon \leq 1.
		\end{equation}
		By Proposition \ref{perturbation-thm}, \eqref{v-exists-condition-1}, \eqref{v-exists-condition-2} and \eqref{v-exists-condition-3}, choosing $\varepsilon$ small enough depending on $\smash{\|u\|_{\L^5_{t,x}}},R_0$ and absolute constants ensures that there exists a unique solution $v(t)$ to \eqref{nls} on $[0,+\infty)$ such that
		$$v\vert_{t=0} = v_0\quad \text{and}\quad \|v\|_{\L^5_{t,x}}+\sup_{t\in [0,+\infty)}\|v(t)\|_{\dot{\H}^{1/2}_x}\leq C_v = C_v(\|u\|_{\L^5_{t,x}},R_0).$$
		Since $\|v_0\|_{\dot{\B}^{1/2}_{2,1}}\leq R_0 + k\varepsilon$, by the induction hypothesis \eqref{assumption-est}, we have
		\begin{equation}\label{v-satisfies-est}
			\|v(t)\|_{\L^{\infty}_x}\leq C(\smash{\|v\|_{\L^5_{t,x}}},R_0,k,\varepsilon)t^{-3/2}\|v_0\|_{\L^1_x}\leq C(\smash{\|u\|_{\L^5_{t,x}}},R_0,k,\varepsilon)t^{-3/2}\|u_0\|_{\L^1_x}.
		\end{equation}
		
		Let $w(t):=u(t)-v(t)$. Then $w(t)$ solves
		\begin{equation}\label{w-solves}
			\i \partial_t w + \Delta w = |u|^2u-|v|^2v,
		\end{equation}
		and by Proposition \ref{perturbation-thm},
		\begin{equation}\label{w-critical-norm-bound}
			\|w\|_{\L^5_{t,x}} + \sup_{t\in [0,+\infty)}\|w(t)\|_{\dot{\H}^{1/2}_x}\leq C_w=C_w(\|u\|_{\L^5_{t,x}},R_0).
		\end{equation}
		
		Since
		\begin{equation}\label{u-infinity-bound}
			\|u(t)\|_{\L^{\infty}_x}\leq \|v(t)\|_{\L^{\infty}_x} + \|w(t)\|_{\L^{\infty}_x}\leq C(\|u\|_{\L^5_{t,x}},R_0,k,\varepsilon)t^{-3/2}\|u_0\|_{\L^1_x} + \|w(t)\|_{\L^{\infty}_x},
		\end{equation}
		it remains to prove
		\begin{equation}\label{w-remain-est}
			\|w(t)\|_{\L^{\infty}_x}\leq C(\|u\|_{\L^5_{t,x}},R_0,k,\varepsilon)t^{-3/2}\|u_0\|_{\L^1_x}
		\end{equation}
		holds for any $t> 0$.
		
		Define
		$$W(t):= \sup_{0<s\leq t} s^{3/2}\|w(s)\|_{\L^{\infty}_x}.$$
		As in the analysis at the beginning of this section, it remains to prove the following bootstrap lemma:
		
		There exists $C_0=C_0(\smash{\|u\|_{\L^5_{t,x}}},R_0,k,\varepsilon)>0$ such that for any $L>0$, if one has $W(L)\leq C_0\|u_0\|_{\L^1_x}$, then one has $W(L)\leq \smash{\tfrac{C_0}{2}}\|u_0\|_{\L^1_x}$.
		
		For $0<t\leq L$, since $w(t)$ solves \eqref{w-solves}, by the Duhamel formula, we can write
		\begin{equation}\label{w-Duhamel-1}
			w(t) = \e^{\i t\Delta}w_0 - \i\int_{0}^{t}\e^{\i(t-s)\Delta}[|u|^2u-|v|^2v](s)\ \mathrm{d}s
		\end{equation}
		Substituting $u=v+w$, we write
		\begin{equation}\label{w-Duhamel-2}
			\begin{aligned}
				w(t) &= \e^{\i t\Delta}w_0 - \i\int_{0}^{t/2}\e^{\i(t-s)\Delta}\left[\mathcal{O}(w^3) + \mathcal{O}(w^2v) + \mathcal{O}(wv^2)\right](s)\ \mathrm{d}s\\
				&\quad -\i\int_{t/2}^{t}\e^{\i(t-s)\Delta}\left[\mathcal{O}(w^3) + \mathcal{O}(w^2v) + \mathcal{O}(wv^2)\right](s)\ \mathrm{d}s\\
				&=: w_{\text{lin}}(t) + G_1(t) + G_2(t).
			\end{aligned}
		\end{equation}
		
		\noindent\textbf{Linear term} $w_{\text{lin}}${\bf.} By dispersive estimate,
		\begin{equation}\label{w_l-est}
			\|w_{\text{lin}}(t)\|_{\L^{\infty}_x}\lesssim t^{-3/2}\|w_0\|_{\L^1_x}\lesssim t^{-3/2}\|u_0\|_{\L^1_x}.
		\end{equation}
		
		\noindent\textbf{Early-time term $G_1(t)$.} As in \eqref{F1-est-1} to \eqref{F1-final-est-2}, we obtain
		\begin{equation}\label{G1-w3-est}
			\left\|\int_{0}^{t/2}\e^{\i(t-s)\Delta}\mathcal{O}(w^3)(s)\ \mathrm{d}s\right\|_{\L^{\infty}_x}\leq \left(C(\|u\|_{\L^5_{t,x}},R_0)+\tfrac{1}{20}C_0\right)t^{-3/2}\|u_0\|_{\L^1_x}.
		\end{equation}
		We next estimate
		\begin{equation}\label{G1-w2v-est-1}
			\left\|\int_{0}^{t/2}\e^{\i(t-s)\Delta}\mathcal{O}(w^2v)(s)\ \mathrm{d}s\right\|_{\L^{\infty}_x}\lesssim t^{-3/2}\|w^2v\|_{\L^1_s\L^1_x((0,+\infty))}.
		\end{equation}
		As in \eqref{F1-est-2}, we introduce $M$ and estimate
		\begin{equation}\label{G1-w2v-est-2}
			\|w^2v\|_{\L^1_s\L^1_x((0,+\infty))}\leq \|w^2v\|_{\L^1_s\L^1_x((0,M))} + \int_{M}^{+\infty}\|w^2v\|_{\L^1_x}\ \mathrm{d}s.
		\end{equation}
		By H\"older and Sobolev embedding,
		\begin{equation}\label{G1-w2v-est-3}
			\|w^2v\|_{\L^1_s\L^1_x((0,M))}\leq M\|w\|^2_{\L^{\infty}_s\L^3_x}\|v\|_{\L^{\infty}_s\L^3_x}\lesssim M\|w\|^2_{\L^{\infty}_s\dot{\H}^{1/2}_x}\|v\|_{\L^{\infty}_s\dot{\H}^{1/2}_x}\leq C(\|u\|_{\L^5_{t,x}},R_0)M.
		\end{equation}
		As in \eqref{u_0-L2-set}, H\"older's inequality and Sobolev embedding give
		$$\|v\|_{\L^2_x}, \|w\|_{\L^2_x}\leq C(\|u\|_{\L^5_{t,x}},R_0)\|u_0\|^{1/4}_{\L^1_x}.$$
		Therefore by bootstrap assumption, 
		\begin{align}\label{G1-w2v-est-4}
			\int_{M}^{+\infty}\|w^2v\|_{\L^1_x}\ \mathrm{d}s&\leq \int_{M}^{+\infty}\|w\|_{\L^2_x}\|v\|_{\L^2_x}\|w\|_{\L^{\infty}_x}\ \mathrm{d}s\leq C_0C(\|u\|_{\L^5_{t,x}},R_0)\|u_0\|^{3/2}_{\L^1_x}\int_{M}^{+\infty}s^{-3/2}\ \mathrm{d}s\notag\\
			&\leq C_0C(\|u\|_{\L^5_{t,x}},R_0)M^{-1/2}\|u_0\|^{3/2}_{\L^1_x}.
		\end{align}
		Choosing $M=C\cdot \|u_0\|_{\L^1_x}$, where $C = C(\smash{\|u\|_{\L^5_{t,x}}},R_0)$ is large enough depending on $\smash{\|u\|_{\L^5_{t,x}}}, R_0$ and some universal constants, by \eqref{G1-w2v-est-1} to \eqref{G1-w2v-est-4}, we obtain
		\begin{equation}\label{G1-w2v-est-5}
			\left\|\int_{0}^{t/2}\e^{\i(t-s)\Delta}\mathcal{O}(w^2v)(s)\ \mathrm{d}s\right\|_{\L^{\infty}_x}\leq \left(C(\|u\|_{\L^5_{t,x}},R_0)+\tfrac{1}{20}C_0\right)t^{-3/2}\|u_0\|_{\L^1_x}.
		\end{equation}
		Similarly,
		\begin{equation}\label{G1-wv2-est-1}
			\left\|\int_{0}^{t/2}\e^{\i(t-s)\Delta}\mathcal{O}(wv^2)(s)\ \mathrm{d}s\right\|_{\L^{\infty}_x}\lesssim t^{-3/2}\|wv^2\|_{\L^1_s\L^1_x((0,+\infty))}
		\end{equation}
		We introduce $M$ and estimate
		\begin{equation}\label{G1-wv2-est-2}
			\|wv^2\|_{\L^1_s\L^1_x((0,+\infty))}\leq \|wv^2\|_{\L^1_s\L^1_x((0,M))} + \int_{M}^{+\infty}\|wv^2\|_{\L^1_x}\ \mathrm{d}s.
		\end{equation}
		By H\"older and Sobolev embedding,
		\begin{equation}\label{G1-wv2-est-3}
			\|wv^2\|_{\L^1_s\L^1_x((0,M))}\leq M\|w\|_{\L^{\infty}_s\L^3_x}\|v\|^2_{\L^{\infty}_s\L^3_x}\lesssim M\|w\|_{\L^{\infty}_s\dot{\H}^{1/2}_x}\|v\|^2_{\L^{\infty}_s\dot{\H}^{1/2}_x} \leq C(\|u\|_{\L^5_{t,x}},R_0)M.
		\end{equation}
		By bootstrap assumption, 
		\begin{align}\label{G1-wv2-est-4}
			\int_{M}^{+\infty}\|wv^2\|_{\L^1_x}\ \mathrm{d}s&\leq \int_{M}^{+\infty}\|v\|^2_{\L^2_x}\|w\|_{\L^{\infty}_x}\ \mathrm{d}s\leq C_0C(\|u\|_{\L^5_{t,x}},R_0)\|u_0\|^{3/2}_{\L^1_x}\int_{M}^{+\infty}s^{-3/2}\ \mathrm{d}s\notag\\
			&\leq C_0C(\|u\|_{\L^5_{t,x}},R_0)M^{-1/2}\|u_0\|^{3/2}_{\L^1_x}.
		\end{align}
		Choosing $M=C\cdot \|u_0\|_{\L^1_x}$, where $C = C(\smash{\|u\|_{\L^5_{t,x}}},R_0)$ is large enough depending on $\smash{\|u\|_{\L^5_{t,x}}}, R_0$ and some universal constants, by \eqref{G1-wv2-est-1} to \eqref{G1-wv2-est-4}, we obtain
		\begin{equation}\label{G1-wv2-est-5}
			\left\|\int_{0}^{t/2}\e^{\i(t-s)\Delta}\mathcal{O}(wv^2)(s)\ \mathrm{d}s\right\|_{\L^{\infty}_x}\leq \left(C(\|u\|_{\L^5_{t,x}},R_0)+\tfrac{1}{20}C_0\right)t^{-3/2}\|u_0\|_{\L^1_x}.
		\end{equation}
		In conclusion, we obtain
		\begin{equation}\label{G1-final-est}
			\|G_1(t)\|_{\L^{\infty}_x}\leq \left(C(\|u\|_{\L^5_{t,x}},R_0)+\tfrac{3}{20}C_0\right)t^{-3/2}\|u_0\|_{\L^1_x}
		\end{equation}
		
		\noindent\textbf{Late-time term $G_2(t)$.} We now estimate $G_2(t)$. As in the base-case estimates, we have
		\begin{equation}\label{G2-w^3-est-1}
			\left\|\int_{t/2}^{t}\e^{\i(t-s)\Delta}\mathcal{O}(w^3)(s)\ \mathrm{d}s\right\|_{\L^{\infty}_x}\leq \sum_{N,N_1:N_1\gtrsim N}\sum_{N_2}\left\|\int_{t/2}^{t}\mathbf{P}_N\left[\e^{\i(t-s)\Delta}\mathcal{O}\left(w_{N_1}w_{N_2}w\right)\right]\ \mathrm{d}s\right\|_{\L^{\infty}_x}.
		\end{equation}
		As in \eqref{F2-term-2-est} to \eqref{F2-term-1-est}, we obtain
		\begin{equation}\label{w^3-est-2}
			\left\|\int_{t/2}^{t}\mathbf{P}_N\left[\e^{\i(t-s)\Delta}\mathcal{O}\left(w_{N_1}w_{N_2}w\right)\right]\ \mathrm{d}s\right\|_{\L^{\infty}_x}\lesssim C_0t^{-3/2}\|u_0\|_{\L^1_x}\left(N^{1/2}\left\|w_{N_1}\right\|_{\L^{\infty}_t\L^2_x}\right)\left(N_2^{1/2}\left\|w_{N_2}\right\|_{\L^{\infty}_t\L^2_x}\right).
		\end{equation}
		Summing over $N$ with $N\lesssim N_1$ firstly, then summing over $N_1,N_2$ and by Proposition \ref{B-stability}, we obtain
		\begin{equation}\label{w^3-est-3}
			\left\|\int_{t/2}^{t}\e^{\i(t-s)\Delta}\mathcal{O}(w^3)(s)\ \mathrm{d}s\right\|_{\L^{\infty}_x}\lesssim C_0 C(\|u\|_{\L^5_{t,x}},R_0)\varepsilon^2t^{-3/2}\|u_0\|_{\L^1_x}\leq \tfrac{1}{20}C_0t^{-3/2}\|u_0\|_{\L^1_x}
		\end{equation}
		for $\varepsilon$ chosen small enough depending on $\smash{\|u\|_{\L^5_{t,x}}},R_0$ and some absolute constants.
		
		By paraproduct decomposition,
		\begin{equation}\label{G2-w2v-term}
			\left\|\int_{t/2}^{t}\e^{\i(t-s)\Delta}\mathcal{O}(w^2v)(s)\ \mathrm{d}s\right\|_{\L^{\infty}_x}
		\end{equation}
		is bounded by
		\begin{equation}\label{w2v-est-1}
			\sum_{N,N_1:N_1\gtrsim N}\sum_{N_2}\left\{\left\|\int_{t/2}^{t}\mathbf{P}_N\left[\e^{\i(t-s)\Delta}\left(w_{N_1}v_{N_2}w\right)\right]\ \mathrm{d}s\right\|_{\L^{\infty}_x} + \left\|\int_{t/2}^{t}\mathbf{P}_N\left[\e^{\i(t-s)\Delta}\left(v_{N_1}w_{N_2}w\right)\right]\ \mathrm{d}s\right\|_{\L^{\infty}_x}\right\}.
		\end{equation}
		Arguing as in \eqref{F2-term-2-est} to \eqref{F2-term-1-est}, and using Proposition \ref{Besov-spacetime-bounds} and Proposition \ref{B-stability}, we obtain
		\begin{align}
			\eqref{G2-w2v-term}&\lesssim \|w\|_{\L^{\infty}_x}\left(\sum_{N}N^{1/2}\left\|w_{N}\right\|_{\L^{\infty}_t\L^2_x}\right)\left(\sum_{N}N^{1/2}\left\|v_{N}\right\|_{\L^{\infty}_t\L^2_x}\right)\notag\\
			&\lesssim \varepsilon C(\|u\|_{\L^5_{t,x}},R_0)C_0t^{-3/2}\|u_0\|_{\L^1_x}\leq\tfrac{1}{20}C_0t^{-3/2}\|u_0\|_{\L^1_x}\label{w2v-est-2}
		\end{align}
		for $\varepsilon$ chosen small enough depending on $\smash{\|u\|_{\L^5_{t,x}}},R_0$ and some absolute constants.
		
		Similarly,
		\begin{equation}\label{G2-wv2-term}
			\left\|\int_{t/2}^{t}\e^{\i(t-s)\Delta}\mathcal{O}(wv^2)(s)\ \mathrm{d}s\right\|_{\L^{\infty}_x}
		\end{equation}
		is bounded by
		\begin{equation}\label{wv2-est-1}
			\sum_{N,N_1:N_1\gtrsim N}\sum_{N_2}\left\{\left\|\int_{t/2}^{t}\mathbf{P}_N\left[\e^{\i(t-s)\Delta}\left(w_{N_1}v_{N_2}v\right)\right]\ \mathrm{d}s\right\|_{\L^{\infty}_x} + \left\|\int_{t/2}^{t}\mathbf{P}_N\left[\e^{\i(t-s)\Delta}\left(v_{N_1}w_{N_2}v\right)\right]\ \mathrm{d}s\right\|_{\L^{\infty}_x}\right\}
		\end{equation}
		and
		\begin{equation}\label{wv2-est-2}
			\eqref{G2-wv2-term}\lesssim \|v\|_{\L^{\infty}_x}\left(\sum_{N}N^{1/2}\left\|w_{N}\right\|_{\L^{\infty}_t\L^2_x}\right)\left(\sum_{N}N^{1/2}\left\|v_{N}\right\|_{\L^{\infty}_t\L^2_x}\right)\leq C(\|u\|_{\L^5_{t,x}},R_0,k,\varepsilon)t^{-3/2}\|u_0\|_{\L^1_x}
		\end{equation}
		by the induction assumption.
		
		By \eqref{w_l-est}, \eqref{G1-final-est}, \eqref{w^3-est-3}, \eqref{w2v-est-2} and \eqref{wv2-est-2}, we obtain
		\begin{equation}\label{w-infinity-final-est}
			\|w(t)\|_{\L^{\infty}_x}\leq \left[C(\|u\|_{\L^5_{t,x}},R_0,k,\varepsilon) + \tfrac{1}{4}C_0\right]\cdot t^{-3/2}\|u_0\|_{\L^1_x}.
		\end{equation}
		Choosing $C_0=20C(\smash{\|u\|_{\L^5_{t,x}}},R_0,k,\varepsilon)$, we close the bootstrap and the proof is complete.
	\end{proof}
	\subsection{Remark on the proof of Theorem \ref{main-result-2}} Since the proof of Theorem \ref{main-result-2} is almost identical to the proof of Theorem \ref{main-result}, we only give a remark here.
	
	Let $u_0\in \smash{\dot{\B}^{2}_{1,1}(\mathbb{R}^3)\cap \L^1_x(\mathbb{R}^3)}$ be radially symmetric and $\smash{\|u_0\|_{\dot{\B}^2_{1,1}}}\leq R_0$. Let $u(t)$ be the corresponding global solution to \eqref{nls} with initial data $u_0$. Firstly we point out that it holds
	\begin{equation}\label{rmk-thm2-1}
		\smash{\||\nabla|^{1/2}u\|_{\L^p_t\L^q_x([0,+\infty))}}\leq C(R_0)
	\end{equation}
	for every Schr\"odinger admissible pair $(p,q)$. By \eqref{scattering-2}, we introduce $\eta>0$ and decompose $[0,+\infty)$ into $J=J(\eta,R_0)$ many intervals $\I_j = [t_j,t_{j+1})$ such that $\smash{\|u\|_{\L^5_{t,x}(\I_j)}}\leq \eta$. By Strichartz estimates and fractional Leibniz rules,
	$$\smash{\||\nabla|^{1/2}u\|_{\L^5_t\L^{30/11}_x(\I_j)}\lesssim \|u_0\|_{\dot{\H}^{1/2}_x} + \|u\|^2_{\L^5_{t,x}(\I_j)}\||\nabla|^{1/2}u\|_{\L^5_t\L^{30/11}_x(\I_j)}\lesssim R_0 + \eta^2\||\nabla|^{1/2}u\|_{\L^5_t\L^{30/11}_x(\I_j)}}.$$
	Choosing $\eta$ small enough depending on some absolute constants, a standard bootstrap argument yields
	$$\smash{\||\nabla|^{1/2}u\|_{\L^5_t\L^{30/11}_x(\I_j)}}\lesssim R_0.$$
	Summing over $j=1,\dots,J(R_0)$, we obtain
	$$\smash{\||\nabla|^{1/2}u\|_{\L^5_t\L^{30/11}_x([0,+\infty))}}\leq C(R_0).$$
	By Strichartz estimates, it follows that
	$$\smash{\||\nabla|^{1/2}u\|_{\L^p_t\L^q_x([0,+\infty))} \lesssim \|u_0\|_{\dot{\H}^{1/2}_x} + \|u\|^2_{\L^5_{t,x}}\||\nabla|^{1/2}u\|_{\L^5_t\L^{30/11}_x([0,+\infty))}}\leq C(R_0),$$
	which gives \eqref{rmk-thm2-1}.
	
	We also have analogues of Propositions 4.3 and 4.4 in \cite{K-2025}, which can be proved in a similar manner to Propositions \ref{Besov-spacetime-bounds} and \ref{B-stability}:
	\begin{prop}\label{Besov-spacetime-bounds-2}
		Assume that $u_0\in \smash{\dot{\B}^{2}_{1,1}(\mathbb{R}^3)}$ is radially symmetric and $\smash{\|u_0\|_{\dot{\B}^{2}_{1,1}}}\leq R_0$ and $u(t)$ is the corresponding global solution to the Cauchy problem \eqref{nls} with initial data $u_0$. Let $(p,q)$ be a Schr\"odinger-admissible pair for $d=3$. Then we have
		\begin{equation}\label{Besov-est-3}
			\sum_{N\in 2^{\mathbb{Z}}}N^{1/2}\left\|u_N\right\|_{\L^p_t\L^q_x}\leq C(R_0).
		\end{equation}
	\end{prop}
	\begin{prop}\label{B-stability-2}
		Let $(p,q)$ be a Schr\"odinger admissible pair for $d=3$. Suppose that $u(t)$ and $v(t)$ are the global solutions to \eqref{nls} with corresponding radially symmetric initial data $u_0,v_0$ respectively obeying the bound
		$$\smash{\|u_0\|_{\dot{\B}^{2}_{1,1}},\ \|v_0\|_{\dot{\B}^{2}_{1,1}}}\leq R_0.$$
		Then
		\begin{equation}\label{Besov-stability-2}
			\sum_{N\in 2^{\mathbb{Z}}}N^{1/2}\|(u-v)_N\|_{\L^p_t\L^q_x}\leq C(R_0)\|u_0-v_0\|_{\dot{\B}^{1/2}_{2,1}}.
		\end{equation}
	\end{prop}
	
	Finally, we remark that the proof of Theorem \ref{main-result-2} does not require any perturbation theorem. When we do induction and decompose the solution as \eqref{decomposition}, $\smash{\|v_0\|_{\dot{\B}^{2}_{1,1}}}\leq R_0 + k\varepsilon\leq R_0 + 1$ ensures the existence of a global solution $v(t)$ to the equation with $v\vert_{t=0} = v_0$ and
	$$\smash{\|v\|_{\L^5_{t,x}([0,+\infty))} + \||\nabla|^{1/2}v\|_{\L^p_t\L^q_x([0,+\infty))}}\leq C(R_0)$$
	for every Schr\"odinger admissible pair $(p,q)$, which is enough to perform the induction.
	
	With these observations, Theorem \ref{main-result-2} can be proved in the same manner as Theorem \ref{main-result}.

\end{document}